\documentclass[12pt]{amsart}

\usepackage[T1]{fontenc}
\usepackage[a4paper,margin=1in]{geometry}
\usepackage{amssymb}  
\usepackage{float}     

\usepackage[colorlinks,
linkcolor=blue,
anchorcolor=blue,
citecolor=blue,
pdfencoding=auto,
unicode,psdextra
]{hyperref}          

\theoremstyle{plain}
\newtheorem{theorem}{Theorem}[section]
\newtheorem{lemma}[theorem]{Lemma}
\newtheorem{proposition}[theorem]{Proposition}
\newtheorem{corollary}[theorem]{Corollary}

\theoremstyle{definition}
\newtheorem{definition}{Definition}
\newtheorem{example}{Example}

\theoremstyle{remark}

\newcommand{\fphi}{\varphi}
\newcommand{\CE}{\operatorname{CE}}
\newcommand{\N}{\mathbb{N}}
\newcommand{\FD}{\mathfrak{F}}
\newcommand{\doi}[1]{\href{https://doi.org/#1}{DOI: #1}}

\author{J. Hamoud}
\address{\textbf{Jasem Hamoud:} Department of Discrete Mathematics, Moscow Institute of Physics and Technology}
\email{hamoud.math@gmail.com}

\title[Duality in Biperiodic Fibonacci Words]{Duality in Biperiodic Fibonacci Words: Substitution, Frequencies and Combinatorial Invariants}

\date{}

\begin{document}

\begin{abstract}
The biperiodic Fibonacci word $\FD^{(a,b)}$ is the standard Sturmian word with directive sequence $(a,b,a,b,\dots)$. 
The words $\FD^{(a,b)}$ and $\FD^{(b,a)}$ are related by the morphism $\sigma_a\colon 0\mapsto 0^a1,\ 1\mapsto 0$ through the exact identity $\FD^{(a,b)}=\sigma_a\bigl(\FD^{(b,a)}\bigr)$. 
This identity gives the letter frequencies, the return words of both letters, the derived sequence with respect to the letter $0$, and the transport of palindromic prefixes from $\FD^{(b,a)}$ to $\FD^{(a,b)}$. 
The slope of $\FD^{(a,b)}$ has the continued fraction expansion $[0;a+1,\overline{b,a}]$, and the critical exponent equals $2+\alpha/\min(a,b)$, where $\alpha$ is the positive root of $t^2-ab\,t-ab=0$. 
The frequency of the letter $1$ and the return words of $1$ depend on the order of $a$ and $b$, whereas the critical exponent, the balance, the abelian complexity and the palindromic complexity do not.
\end{abstract}

\maketitle

\noindent\rule{15.9cm}{1.0pt}

\noindent
\textbf{
Keywords:} Fibonacci word; Sturmian word; biperiodic Fibonacci word; standard sequence; continued fraction; critical exponent; return words; palindromic complexity; combinatorics on words.

\medskip

\noindent
{\bf
MSC 2020:} 68R15 (Primary); 11B39, 37B10, 11A55 (Secondary).

\medskip

\noindent
{\bf
UDC:} 519.1

\noindent\rule{15.9cm}{1.0pt}

\section{Introduction}\label{sec01intro}

Let $\Sigma=\{0,1\}$ be a binary alphabet. A finite word over $\Sigma$ is an
element of the free monoid $\Sigma^{*}$, and an infinite word is an element of
$\Sigma^{\N}$. For a finite word $w=w_0w_1\cdots w_{k-1}$ we write $|w|=k$ for
its length and $|w|_c$ for the number of occurrences of the letter $c$ in $w$,
so that $|w|=|w|_0+|w|_1$. A finite word $u$ is a factor of a finite or
infinite word $w$ if $u=w_iw_{i+1}\cdots w_{i+|u|-1}$ for some $i\geqslant 0$,
and a prefix if $i=0$. The set of all factors of $w$ is denoted by
$\mathcal{L}(w)$, the set of factors of length $n$ by $\mathcal{L}_n(w)$, and
the factor complexity function of $w$ is $p_w(n)=|\mathcal{L}_n(w)|$. A word is
a palindrome if it coincides with its reversal $\widetilde{w}$. A morphism
$\tau\colon\Sigma^{*}\to\Sigma^{*}$ is a monoid homomorphism, determined by the
two words $\tau(0)$ and $\tau(1)$, and it acts on infinite words letter by
letter whenever the image is non-erasing.

The Fibonacci word is the fixed point
\begin{equation}\label{eqq01fibword}
	\FD = 0100101001001010010100\cdots
\end{equation}
of the morphism
\begin{equation}\label{eqq01fibmorph}
	\sigma\colon 0\mapsto 01,\qquad 1\mapsto 0.
\end{equation}
It is a Sturmian word, that is, an aperiodic binary word in which the numbers
of ones in any two factors of equal length differ by at most one~\cite{MorseHedlund1940}.
Among Sturmian words it is the least repetitive. Its critical exponent, the
supremum of $|u|/\pi(u)$ over all nonempty factors $u$, where $\pi(u)$ denotes
the smallest period of $u$, is
\begin{equation}\label{eqq01cefib}
	\CE(\FD)=2+\fphi, \qquad \fphi=\frac{1+\sqrt5}{2},
\end{equation}
as shown in~\cite{MignosiPirillo1992} and revisited in~\cite{Cassaigne2008, MousaviShallit2015}.
For a general Sturmian word the critical exponent is expressed through the
continued fraction expansion of its slope~\cite{DamanikLenz2002}, and this
description underlies the whole of the present paper.

The biperiodic Fibonacci words were introduced by Ram\'irez and
Rubiano~\cite{Ramirez2015} as a two-parameter extension of the Fibonacci word,
see also~\cite{EdsonYayenie2009} for the corresponding numerical sequences.
Fix integers $a,b\geqslant 1$ and define finite words $\FD^{(a,b)}_n$ by
\begin{equation}\label{eqq01seeds}
	\FD^{(a,b)}_0 = 1, \qquad \FD^{(a,b)}_1 = 0,
\end{equation}
and, for $n\geqslant 2$, by the recurrence
\begin{equation}\label{eqq01biperiodicrecurrence}
	\FD^{(a,b)}_n =
	\begin{cases}
		\bigl(\FD^{(a,b)}_{n-1}\bigr)^a\, \FD^{(a,b)}_{n-2} & \text{if } n \text{ is even},\\[4pt]
		\bigl(\FD^{(a,b)}_{n-1}\bigr)^b\, \FD^{(a,b)}_{n-2} & \text{if } n \text{ is odd}.
	\end{cases}
\end{equation}
For every $n\geqslant 1$ the word $\FD^{(a,b)}_n$ is a prefix of $\FD^{(a,b)}_{n+1}$ and $|\FD^{(a,b)}_n|\to\infty$ (see Lemma~\ref{lem03prefix} and Proposition~\ref{prop01fabisstandard}), hence the
sequence $(\FD^{(a,b)}_n)_{n\geqslant 1}$ converges to an infinite word $\FD^{(a,b)}$, the biperiodic Fibonacci word. For $a=b=1$ the recurrence~\eqref{eqq01biperiodicrecurrence} is the Fibonacci recurrence and $\FD^{(1,1)}$ is the word in~\eqref{eqq01fibword}.

Two quantities attached to $\FD^{(a,b)}$ are governed by the algebraic number
\begin{equation}\label{eqq01alphadef}
	\alpha=\alpha(a,b)=\frac{ab+\sqrt{a^2b^2+4ab}}{2},
\end{equation}
which is the positive root of $t^2-ab\,t-ab=0$ and satisfies
$\alpha(a,b)=\alpha(b,a)$. The first quantity is the pair of letter
frequencies, proved in Theorem~\ref{thm02freqduality} to be
\begin{equation}\label{eqq01freqs}
	\lim_{n\to\infty}\frac{|\FD^{(a,b)}_n|_1}{|\FD^{(a,b)}_n|}=\frac{b}{\alpha+b},
	\qquad
	\lim_{n\to\infty}\frac{|\FD^{(a,b)}_n|_0}{|\FD^{(a,b)}_n|}=\frac{\alpha}{\alpha+b}.
\end{equation}
The second quantity is the critical exponent, which by the
description of~\cite{DamanikLenz2002} and the continued fraction computed in
Section~\ref{sec01continuedfraction} takes the form
\begin{equation}\label{eqq01cebiperiodic}
	\CE\bigl(\FD^{(a,b)}\bigr)=2+\frac{\alpha}{\min(a,b)}.
\end{equation}
For $a=b=1$ we have $\alpha=\fphi$, and both~\eqref{eqq01freqs}
and~\eqref{eqq01cebiperiodic} reduce to the classical values, namely the
frequency $1/\fphi^{2}$ of the letter $1$ and the exponent in~\eqref{eqq01cefib}.

The two formulas behave differently under the exchange of $a$ and $b$. The
exponent~\eqref{eqq01cebiperiodic} is symmetric in $a$ and $b$, whereas the
frequency of the letter $1$ in~\eqref{eqq01freqs} is not. For instance, $\FD^{(1,2)}$ and $\FD^{(2,1)}$ have the same critical exponent but different frequencies of the letter $1$ (see Table~\ref{tab05values}).
This contrast raises the question of how $\FD^{(a,b)}$ and $\FD^{(b,a)}$ are related at the level of the words themselves. The answer is given by the
morphism
\begin{equation}\label{eqq01sigmaa}
	\sigma_a\colon 0\mapsto 0^a1,\qquad 1\mapsto 0,
\end{equation}
which satisfies the identity
\begin{equation}\label{eqq01duality}
	\FD^{(a,b)}=\sigma_a\bigl(\FD^{(b,a)}\bigr)
\end{equation}
exactly, without any relabelling of letters and without any prefix correction.
We call~\eqref{eqq01duality} the duality between $\FD^{(a,b)}$ and
$\FD^{(b,a)}$. 
It is an instance of the desubstitution property of standard Sturmian words~\cite{Lothaire2002, AllouecheShallit2003}.
We use it to compare $\FD^{(a,b)}$ and $\FD^{(b,a)}$ systematically: it yields explicit formulas for the letter frequencies, the return words and the derived sequence, and it shows which invariants are preserved by the exchange of $a$ and $b$ and which are not.

The ingredients are classical: the identification of $\FD^{(a,b)}$ as a
standard Sturmian word and the formula of Damanik and Lenz for the critical
exponent. The contribution of the paper is the explicit description of the
biperiodic family that results from them, namely the following. First,
$\FD^{(a,b)}$ is the standard Sturmian word with directive sequence
$(a,b,a,b,\dots)$, and the identity~\eqref{eqq01duality} holds exactly.
Second, the letter frequencies are
given by~\eqref{eqq01freqs}, and the slope of $\FD^{(a,b)}$ has the continued
fraction expansion
\begin{equation}\label{eqq01slopecf}
	\theta^{(a,b)}=\bigl[\,0;\,a+1,\,\overline{b,\,a}\,\bigr].
\end{equation}
Third, the return words of the letter $0$ are $0$ and $01$ for all $a,b$,
while those of the letter $1$ are $10^a$ and $10^{a+1}$, so that the exchange of
$a$ and $b$ acts on them by a change of exponent. Fourth, $\FD^{(a,b)}$ has
palindromic prefixes of arbitrary length, and the critical exponent is given
by~\eqref{eqq01cebiperiodic}. The geometric side of the theory, namely the
fractal curves and the biperiodic Fibonacci snowflakes, was developed
in~\cite{Ramirez2015, FibSnowflake2024, Dumaine2009} and is not treated here.

The paper is organised as follows. Section~\ref{sec02prelim} fixes the notation
and recalls the definitions of critical exponent, balance, abelian complexity
and palindromic complexity. 
Section~\ref{sec01sadic} identifies $\FD^{(a,b)}$ with the standard Sturmian word of directive sequence $(a,b,a,b,\dots)$, proves the identity~\eqref{eqq01duality} by desubstitution, and relates $\alpha$ to the periodic continued fraction $[a;b,a,b,\dots]$.
Section~\ref{sec04ceduality} derives the letter frequencies,
Section~\ref{sec5duality} describes the return words and the derived sequence, and Section~\ref{sec01palduality} treats palindromic prefixes.
Section~\ref{sec01continuedfraction} computes the continued fraction expansion of the slope and the critical exponent.
%======================
\section{Preliminaries}\label{sec02prelim}

A finite word $w\neq\varepsilon$ has period $p\geqslant 1$ if $w_i=w_{i+p}$ for
all admissible indices $i$, and $\pi(w)$ denotes the smallest period of $w$. A
nonempty word $u$ is primitive if $u=v^k$ implies $k=1$. The exponent of $w$ is
the rational number
\begin{equation}\label{eqq02exponent}
	e(w)=\frac{|w|}{\pi(w)}.
\end{equation}
A word of the form $u^ju_p$, where $u$ is primitive, $j\geqslant 1$ and $u_p$ is a
proper prefix of $u$, has length $j|u|+|u_p|$, and its exponent
in~\eqref{eqq02exponent} equals $j+|u_p|/|u|$ whenever $|u|$ is its smallest period.

\begin{definition}\label{def01ce}
The critical exponent of an infinite word $w$ is
	\begin{equation}\label{eqq02ce}
		\CE(w)=\sup\{\,e(u) : u\in\mathcal{L}(w)\setminus\{\varepsilon\}\,\}\in(1,\infty].
	\end{equation}
\end{definition}

\begin{definition}\label{def01balance}
	For an infinite word $w$ the balance function is
	\begin{equation}\label{eqq02balance}
		B_w(n)=\max\bigl\{\,\bigl||u|_1-|v|_1\bigr|:\ u,v\in\mathcal{L}_n(w)\,\bigr\}.
	\end{equation}
	The word $w$ is $C$-balanced if $B_w(n)\leqslant C$ for every $n\geqslant 1$.
\end{definition}

An infinite binary word is Sturmian if it is aperiodic and $1$-balanced in the
sense of~\eqref{eqq02balance}. By the theorem of Morse and
Hedlund~\cite{MorseHedlund1940}, aperiodic and $1$-balanced words are exactly
the words with factor complexity $p_w(n)=n+1$ for every $n\geqslant 0$. Every Sturmian word has a slope: the limit
\begin{equation}\label{eqq02slope}
	\theta(w)=\lim_{n\to\infty}\frac{|w_0w_1\cdots w_{n-1}|_1}{n}
\end{equation}
exists and is an irrational number in $(0,1)$~\cite{Lothaire2002}.
The slope in~\eqref{eqq02slope} is the frequency of the letter $1$ in $w$.

\begin{definition}\label{def01abelian}
	Two words $u,v\in\mathcal{L}_n(w)$ are abelian equivalent if $|u|_c=|v|_c$ for
	each letter $c\in\Sigma$. The abelian complexity function of $w$ is
	\begin{equation}\label{eqq02abelian}
		\mathcal{AC}_w(n)=\bigl|\mathcal{L}_n(w)/\!\sim_{ab}\bigr|,
	\end{equation}
	the number of abelian equivalence classes of factors of length $n$.
\end{definition}

\begin{definition}\label{def01palcomplexity}
	The palindromic complexity function of $w$ is
	\begin{equation}\label{eqq02pal}
		P_w(n)=\bigl|\{\,u\in\mathcal{L}_n(w) : u=\widetilde u\,\}\bigr|.
	\end{equation}
\end{definition}

\begin{definition}\label{def01central}
	For a word $w$ with $|w|\geqslant 2$ let $w^{-}$ be the word obtained from $w$ by
	removing its last two letters, that is,
	\begin{equation}\label{eqq02minus}
		w^{-}=w_0w_1\cdots w_{|w|-3}.
	\end{equation}
\end{definition}

\begin{definition}\label{def01returnword}
	Let $w$ be an infinite word and let $u\in\mathcal{L}(w)$ occur in $w$ at the
	positions $p_0<p_1<p_2<\cdots$. The return words of $u$ in $w$ are the factors
	\begin{equation}\label{eqq02return}
		w[p_i,p_{i+1})=w_{p_i}w_{p_i+1}\cdots w_{p_{i+1}-1},\qquad i\geqslant 0.
	\end{equation}
\end{definition}

The following characterization of Sturmian words by return words is due to
Vuillon~\cite{Vuillon2001}.

\begin{theorem}\label{thm01vuillon}
	An aperiodic binary word $w$ is Sturmian if and only if every
	$u\in\mathcal{L}(w)$ has exactly two return words in $w$.
\end{theorem}

The remaining invariants of Sturmian words are known in closed form. If $w$ is
Sturmian then its abelian complexity satisfies
\begin{equation}\label{eqq02abelianvalue}
	\mathcal{AC}_w(n)=2\qquad\text{for all } n\geqslant 1,
\end{equation}
as shown by Coven and Hedlund~\cite{CovenHedlund1973}, and its palindromic
complexity satisfies
\begin{equation}\label{eqq02palvalue}
	P_w(n)=\begin{cases}1 & \text{if } n \text{ is even},\\ 2 & \text{if } n \text{ is odd},\end{cases}
\end{equation}
as shown by Droubay and Pirillo~\cite{DroubayPirillo1999}. Both values
in~\eqref{eqq02abelianvalue} and~\eqref{eqq02palvalue} are independent of the
slope. The critical exponent, on the other hand, depends on the continued
fraction expansion of the slope, according to the theorem of Damanik and
Lenz~\cite{DamanikLenz2002}.

\begin{theorem}\label{thm01damaniklenz}
	Let $w$ be a Sturmian word with slope $\theta(w)=[0;a_1,a_2,a_3,\dots]$. Then
	\begin{equation}\label{eqq02damaniklenz}
		\CE(w)=2+\limsup_{n\to\infty}\,[a_{n+1};a_n,a_{n-1},\dots,a_1].
	\end{equation}
\end{theorem}

Formula~\eqref{eqq02damaniklenz} is used in Section~\ref{sec01continuedfraction}
to compute $\CE\bigl(\FD^{(a,b)}\bigr)$. The value $\CE(\FD)=2+\fphi$ recalled
in~\eqref{eqq01cefib} is the case in which all partial quotients equal $1$.
%========== SECTION 3 =================================== 
%========================================================
\section{Standard Sequences and the Desubstitution Duality}\label{sec01sadic}

The recurrence~\eqref{eqq01biperiodicrecurrence} has the shape of the classical
construction of standard Sturmian words from continued fractions. This section
places $\FD^{(a,b)}$ in that framework. The main outcome is the exact relation
between $\FD^{(a,b)}$ and $\FD^{(b,a)}$ given by the morphism $\sigma_a$. The
section also records what follows for balance, abelian complexity and
palindromic complexity, and it identifies the number $\alpha$ as a value of a
periodic continued fraction. The results proved here are the tools used in
Sections~\ref{sec04ceduality} and~\ref{sec01continuedfraction}.

For an integer $d\geqslant 1$ the elementary morphism $\sigma_d$ is defined by
\begin{equation}\label{eqq03sigmad}
	\sigma_d(0)=0^d1,\qquad \sigma_d(1)=0.
\end{equation}
It is non-erasing, and $|\sigma_d(x)|\geqslant|x|$ for every finite word $x$. If
$x$ is a prefix of $y$ then $\sigma_d(x)$ is a prefix of $\sigma_d(y)$. Hence
$\sigma_d$ extends to infinite words by
\begin{equation}\label{eqq03extend}
	\sigma_d(w)=\lim_{n\to\infty}\sigma_d(w_0w_1\cdots w_{n-1}),
\end{equation}
and the extension is continuous in the prefix topology. The next lemma shows
that the extension~\eqref{eqq03extend} is injective, which is needed to speak of
the preimage of an infinite word.

\begin{lemma}\label{lem03injective}
	For every $d\geqslant 1$ the map $\sigma_d$ is injective on infinite words.
\end{lemma}
\begin{proof}
	Let $w,w'$ be infinite words with $\sigma_d(w)=\sigma_d(w')$. If $w$ begins with
	$0$ then $\sigma_d(w)$ begins with $0^d1$, and the first letter $1$ of
	$\sigma_d(w)$ is at position $d$. If $w'$ begins with $1^j0$ for some $j\geqslant 1$
	then the first letter $1$ of $\sigma_d(w')$ is at position $j+d>d$, and if
	$w'=1^\infty$ then $\sigma_d(w')$ contains no letter $1$. In both cases
	$\sigma_d(w)\neq\sigma_d(w')$, so $w$ and $w'$ begin with the same letter. Removing
	the common block $\sigma_d(w_0)=\sigma_d(w'_0)$ from the two images gives
	$\sigma_d(w_1w_2\cdots)=\sigma_d(w'_1w'_2\cdots)$, and induction on the position
	shows $w=w'$.
\end{proof}

\begin{definition}\label{def01standardsequence}
	Let $(d_k)_{k\geqslant 1}$ be a sequence of positive integers. The associated
	standard sequence $(s_k)_{k\geqslant -1}$ is defined by
	\begin{equation}\label{eqq03standard}
		s_{-1}=1,\qquad s_0=0,\qquad s_k=s_{k-1}^{\,d_k}\,s_{k-2}\quad\text{for } k\geqslant 1.
	\end{equation}
\end{definition}

The next lemma collects the elementary properties of~\eqref{eqq03standard}. Its
second part makes the limit word in~\eqref{eqq03limit} well defined, and its
third part gives the recurrences for lengths and letter counts used in
Section~\ref{sec04ceduality}.

\begin{lemma}\label{lem03prefix}
	Let $(s_k)$ be the standard sequence of $(d_k)$. Then:
	\begin{enumerate}
		\item $s_k$ is a prefix of $s_{k+1}$ for every $k\geqslant 0$;
		\item $|s_k|\to\infty$ as $k\to\infty$;
		\item $L_k=|s_k|$ and $N_k=|s_k|_1$ satisfy
		\begin{equation}\label{eqq03lengths}
			L_k=d_kL_{k-1}+L_{k-2},\qquad L_{-1}=L_0=1,
		\end{equation}
		\begin{equation}\label{eqq03ones}
			N_k=d_kN_{k-1}+N_{k-2},\qquad N_{-1}=1,\ N_0=0,
		\end{equation}
		for every $k\geqslant 1$.
	\end{enumerate}
\end{lemma}
\begin{proof}
	For $k\geqslant 0$ the word $s_{k+1}=s_k^{d_{k+1}}s_{k-1}$ begins with $s_k$ since
	$d_{k+1}\geqslant 1$. This gives the first part. The recurrences~\eqref{eqq03lengths}
	and~\eqref{eqq03ones} follow by counting letters in~\eqref{eqq03standard}. By
	\eqref{eqq03lengths} we have $L_{k+1}\geqslant L_k+L_{k-1}>L_k$ for $k\geqslant 0$,
	which gives the second part.
\end{proof}

By Lemma~\ref{lem03prefix} the sequence $(s_k)_{k\geqslant 0}$ is an increasing chain
of prefixes with unbounded lengths, and therefore converges to an infinite word
\begin{equation}\label{eqq03limit}
	s^{(d_1,d_2,d_3,\dots)}=\lim_{k\to\infty}s_k\in\Sigma^{\N}.
\end{equation}
The following classical theorem describes the words arising
from~\eqref{eqq03limit}.

\begin{theorem}[\cite{Lothaire2002}, Chapter 2]\label{thm01classical}
	For every sequence $(d_k)_{k\geqslant 1}$ of positive integers the word
	$s^{(d_1,d_2,\dots)}$ is a standard Sturmian word with slope
	\begin{equation}\label{eqq03classicalslope}
		\theta=[0;\,d_1+1,\,d_2,\,d_3,\dots].
	\end{equation}
	Every standard Sturmian word is obtained in this way. The slope is a quadratic
	irrational if and only if $(d_k)$ is eventually periodic.
\end{theorem}

The next lemma is the finite version of the desubstitution property of standard
sequences. It compares the standard sequence of $(d_1,d_2,\dots)$ with the
standard sequence of the shifted sequence $(d_2,d_3,\dots)$.

\begin{lemma}\label{lem03finitedesub}
	Let $(s_k)_{k\geqslant -1}$ and $(t_k)_{k\geqslant -1}$ be the standard sequences of
	$(d_1,d_2,d_3,\dots)$ and $(d_2,d_3,d_4,\dots)$ respectively. Then
	\begin{equation}\label{eqq03finitedesub}
		\sigma_{d_1}(t_k)=s_{k+1}\qquad\text{for every } k\geqslant -1.
	\end{equation}
\end{lemma}
\begin{proof}
	For $k=-1$ and $k=0$ we have $\sigma_{d_1}(t_{-1})=\sigma_{d_1}(1)=0=s_0$ and
	$\sigma_{d_1}(t_0)=\sigma_{d_1}(0)=0^{d_1}1=s_1$. Let $k\geqslant 1$ and assume
	\eqref{eqq03finitedesub} for $k-1$ and $k-2$. The recurrence~\eqref{eqq03standard}
	for the shifted sequence reads $t_k=t_{k-1}^{\,d_{k+1}}t_{k-2}$. Since
	$\sigma_{d_1}$ is a morphism,
	\begin{equation}\label{eqq03inductionstep}
		\sigma_{d_1}(t_k)=\sigma_{d_1}(t_{k-1})^{d_{k+1}}\sigma_{d_1}(t_{k-2})=s_k^{\,d_{k+1}}s_{k-1}=s_{k+1},
	\end{equation}
	where the last equality is~\eqref{eqq03standard} at index $k+1$.
\end{proof}

Passing to the limit in~\eqref{eqq03finitedesub} gives the desubstitution
identity for infinite words.

\begin{theorem}\label{thm03desubgeneral}
	For every sequence $(d_k)_{k\geqslant 1}$ of positive integers,
	\begin{equation}\label{eqq03desubgeneral}
		s^{(d_1,d_2,d_3,\dots)}=\sigma_{d_1}\bigl(s^{(d_2,d_3,\dots)}\bigr).
	\end{equation}
\end{theorem}
\begin{proof}
	By Lemma~\ref{lem03prefix} the words $t_k$ are prefixes of
	$s^{(d_2,d_3,\dots)}$ of unbounded length. By~\eqref{eqq03extend}
	and~\eqref{eqq03finitedesub},
	\begin{equation}\label{eqq03limitdesub}
		\sigma_{d_1}\bigl(s^{(d_2,d_3,\dots)}\bigr)=\lim_{k\to\infty}\sigma_{d_1}(t_k)=\lim_{k\to\infty}s_{k+1}=s^{(d_1,d_2,d_3,\dots)}.
	\end{equation}
\end{proof}

The next proposition identifies the biperiodic Fibonacci words with the standard
sequences of the alternating directive sequence. It is the link between the
recurrence~\eqref{eqq01biperiodicrecurrence} and the general theory above.

\begin{proposition}\label{prop01fabisstandard}
	Let $a,b\geqslant 1$ and let $(s_k)$ be the standard sequence of
	\begin{equation}\label{eqq03directive}
		(d_1,d_2,d_3,d_4,\dots)=(a,b,a,b,\dots),\qquad d_k=\begin{cases}a & k\text{ odd},\\ b & k\text{ even}.\end{cases}
	\end{equation}
	Then $\FD^{(a,b)}_n=s_{n-1}$ for every $n\geqslant 0$, and
	$\FD^{(a,b)}=s^{(a,b,a,b,\dots)}$.
\end{proposition}
\begin{proof}
	The seeds~\eqref{eqq01seeds} give $\FD^{(a,b)}_0=1=s_{-1}$ and
	$\FD^{(a,b)}_1=0=s_0$. Let $n\geqslant 2$ and assume the claim for $n-1$ and $n-2$.
	The exponent in~\eqref{eqq01biperiodicrecurrence} is $a$ for even $n$ and $b$ for
	odd $n$, and by~\eqref{eqq03directive} this is $d_{n-1}$. Hence
	\begin{equation}\label{eqq03identification}
		\FD^{(a,b)}_n=\bigl(\FD^{(a,b)}_{n-1}\bigr)^{d_{n-1}}\FD^{(a,b)}_{n-2}=s_{n-2}^{\,d_{n-1}}s_{n-3}=s_{n-1}.
	\end{equation}
	The limit statement follows from~\eqref{eqq03limit}.
\end{proof}

The first main result of the paper follows by combining
Theorem~\ref{thm03desubgeneral} and Proposition~\ref{prop01fabisstandard}. The
shifted directive sequence of $(a,b,a,b,\dots)$ is $(b,a,b,a,\dots)$, which is
the directive sequence of $\FD^{(b,a)}$.

\begin{theorem}\label{thm01desub}
	For all $a,b\geqslant 1$ and all $n\geqslant 0$,
	\begin{equation}\label{eqq03dualfinite}
		\sigma_a\bigl(\FD^{(b,a)}_n\bigr)=\FD^{(a,b)}_{n+1},
	\end{equation}
	and consequently
	\begin{equation}\label{eqq03dualinfinite}
		\sigma_a\bigl(\FD^{(b,a)}\bigr)=\FD^{(a,b)}.
	\end{equation}
\end{theorem}
\begin{proof}
	Let $(s_k)$ and $(t_k)$ be the standard sequences of $(a,b,a,b,\dots)$ and
	$(b,a,b,a,\dots)$. By Proposition~\ref{prop01fabisstandard} applied to the pairs
	$(a,b)$ and $(b,a)$ we have $s_{n}=\FD^{(a,b)}_{n+1}$ and
	$t_{n}=\FD^{(b,a)}_{n+1}$ for every $n\geqslant -1$. Lemma~\ref{lem03finitedesub}
	with $d_1=a$ gives $\sigma_a(t_k)=s_{k+1}$, which is~\eqref{eqq03dualfinite} with
	$n=k+1$. Theorem~\ref{thm03desubgeneral} gives~\eqref{eqq03dualinfinite}.
\end{proof}

The relation~\eqref{eqq03dualinfinite} determines $\FD^{(b,a)}$ uniquely.

\begin{corollary}\label{cor03uniquepreimage}
	The word $\FD^{(b,a)}$ is the only infinite word $w$ with
	$\sigma_a(w)=\FD^{(a,b)}$.
\end{corollary}
\begin{proof}
	This is Lemma~\ref{lem03injective} together with~\eqref{eqq03dualinfinite}.
\end{proof}

Theorem~\ref{thm01classical} and Proposition~\ref{prop01fabisstandard} also give
the structure of $\FD^{(a,b)}$ as a Sturmian word.

\begin{corollary}\label{cor01sturmian}
	For all $a,b\geqslant 1$ the word $\FD^{(a,b)}$ is a standard Sturmian word with
	slope
	\begin{equation}\label{eqq03slope}
		\theta\bigl(\FD^{(a,b)}\bigr)=[0;\,a+1,\,b,\,a,\,b,\,a,\dots],
	\end{equation}
	which is a quadratic irrational.
\end{corollary}
\begin{proof}
	The directive sequence~\eqref{eqq03directive} is periodic. Apply
	Theorem~\ref{thm01classical} to it and use
	Proposition~\ref{prop01fabisstandard}.
\end{proof}

Section~\ref{sec02prelim} recalled the values of three invariants for every
Sturmian word. Corollary~\ref{cor01sturmian} makes them available for
$\FD^{(a,b)}$.

\begin{corollary}\label{cor03invariants}
	For all $a,b\geqslant 1$ and all $n\geqslant 1$ the word $\FD^{(a,b)}$ satisfies
	\begin{equation}\label{eqq03balabel}
		B_{\FD^{(a,b)}}(n)=1,\qquad \mathcal{AC}_{\FD^{(a,b)}}(n)=2,
	\end{equation}
	\begin{equation}\label{eqq03palindromic}
		P_{\FD^{(a,b)}}(n)=\begin{cases}1 & n\text{ even},\\ 2 & n\text{ odd}.\end{cases}
	\end{equation}
	In particular these three invariants coincide for $\FD^{(a,b)}$ and
	$\FD^{(b,a)}$.
\end{corollary}
\begin{proof}
	By Corollary~\ref{cor01sturmian} the word $\FD^{(a,b)}$ is Sturmian, hence
	$1$-balanced, so $B_{\FD^{(a,b)}}(n)\leqslant 1$ by~\eqref{eqq02balance}. The
	value $\mathcal{AC}_{\FD^{(a,b)}}(n)=2$ is~\eqref{eqq02abelianvalue}. Two
	distinct abelian classes of factors of the same length $n$ have distinct
	numbers of letters $1$, which gives $B_{\FD^{(a,b)}}(n)\geqslant 1$.
	Formula~\eqref{eqq03palindromic} is~\eqref{eqq02palvalue}.
\end{proof}

The remaining part of the section connects $\alpha$ with the directive
sequence. For $a,b\geqslant 1$ let
\begin{equation}\label{eqq03chi}
	\chi(a,b)=[a;\,b,\,a,\,b,\dots]
\end{equation}
be the value of the purely periodic continued fraction with period $(a,b)$.

\begin{proposition}\label{prop01AlphAacf}
	For all $a,b\geqslant 1$,
	\begin{equation}\label{eqq03alphachi}
		\alpha(a,b)=b\,\chi(a,b).
	\end{equation}
\end{proposition}
\begin{proof}
	Let $x=\chi(a,b)>0$. By periodicity $x=a+1/y$ with $y=[b;a,b,\dots]=b+1/x$, hence
	\begin{equation}\label{eqq03fixed}
		x=a+\cfrac{1}{b+\cfrac{1}{x}}=a+\frac{x}{bx+1}.
	\end{equation}
	Clearing denominators in~\eqref{eqq03fixed} gives
	\begin{equation}\label{eqq03quadx}
		bx^2-abx-a=0.
	\end{equation}
	Multiplying~\eqref{eqq03quadx} by $b$ and writing $y=bx$ gives
	\begin{equation}\label{eqq03quady}
		y^2-ab\,y-ab=0.
	\end{equation}
	The product of the roots of~\eqref{eqq03quady} is $-ab<0$, hence exactly one root
	is positive, and $y=bx>0$ is that root. By~\eqref{eqq01alphadef} the positive
	root is $\alpha(a,b)$.
\end{proof}

Proposition~\ref{prop01AlphAacf} applied to both orders of the pair, combined
with~\eqref{eqq03slope}, expresses the slope of $\FD^{(a,b)}$ through $\alpha$.

\begin{corollary}\label{cor03slopealpha}
	For all $a,b\geqslant 1$,
	\begin{equation}\label{eqq03chialpha}
		\chi(a,b)=\frac{\alpha}{b},\qquad \chi(b,a)=\frac{\alpha}{a},
	\end{equation}
	and the slope of $\FD^{(a,b)}$ equals
	\begin{equation}\label{eqq03slopealpha}
		\theta\bigl(\FD^{(a,b)}\bigr)=\frac{1}{1+\chi(a,b)}=\frac{b}{\alpha+b}.
	\end{equation}
\end{corollary}
\begin{proof}
	The first formula in~\eqref{eqq03chialpha} is~\eqref{eqq03alphachi}. The second
	is the same formula for the pair $(b,a)$, using $\alpha(b,a)=\alpha(a,b)$. Since
	$1+\chi(a,b)=[a+1;\,b,\,a,\,b,\dots]$, the reciprocal of $1+\chi(a,b)$ is the
	continued fraction in~\eqref{eqq03slope}. The last equality
	in~\eqref{eqq03slopealpha} follows from the first formula
	in~\eqref{eqq03chialpha}.
\end{proof}

By~\eqref{eqq02slope} the slope is the frequency of the letter $1$, so
Corollary~\ref{cor03slopealpha} already indicates the value of the frequency in
\eqref{eqq01freqs}. Section~\ref{sec04ceduality} proves this value directly from
the morphism $\sigma_a$, and Section~\ref{sec01continuedfraction} uses
\eqref{eqq03chialpha} to compute the critical exponent.

%==========================comlete steps ==================
%==========================================================
\section{Letter Frequencies under the Duality}\label{sec04ceduality}

This section studies how the morphism $\sigma_a$ of Theorem~\ref{thm01desub} transports letter counts between $\FD^{(b,a)}$ and $\FD^{(a,b)}$. 
The letter frequencies are computed directly from the identity~\eqref{eqq03dualinfinite}, which gives a second derivation of the value already indicated by Corollary~\ref{cor03slopealpha}.

The first lemma records how $\sigma_d$ acts on the number of occurrences of each letter.

\begin{lemma}\label{lem01length}
For every finite word $x\in\Sigma^{*}$ and every $d\geqslant 1$,
	\begin{equation}\label{eqq04letters}
		|\sigma_d(x)|_1=|x|_0,\qquad |\sigma_d(x)|_0=d\,|x|_0+|x|_1,
	\end{equation}
	and consequently
	\begin{equation}\label{eqq04length}
		|\sigma_d(x)|=d\,|x|_0+|x|.
	\end{equation}
\end{lemma}
\begin{proof}
	The word $\sigma_d(0)=0^d1$ contains one letter $1$ and $d$ letters $0$, and
	$\sigma_d(1)=0$ contains one letter $0$. Since $\sigma_d$ is a morphism, the
	counts add over the letters of $x$, which gives~\eqref{eqq04letters}. Adding the
	two identities in~\eqref{eqq04letters} gives $|\sigma_d(x)|=(d+1)|x|_0+|x|_1$,
	which is~\eqref{eqq04length}.
\end{proof}

The letter frequencies of $\FD^{(a,b)}$ are now computed from the finite
identities~\eqref{eqq03dualfinite}.

\begin{theorem}\label{thm02freqduality}
	For all $a,b\geqslant 1$ the limits
	\begin{equation}\label{eqq04freqs}
		\operatorname{freq}_1\bigl(\FD^{(a,b)}\bigr)=\lim_{n\to\infty}\frac{|\FD^{(a,b)}_n|_1}{|\FD^{(a,b)}_n|}=\frac{b}{\alpha+b},
		\qquad
		\operatorname{freq}_0\bigl(\FD^{(a,b)}\bigr)=\lim_{n\to\infty}\frac{|\FD^{(a,b)}_n|_0}{|\FD^{(a,b)}_n|}=\frac{\alpha}{\alpha+b}
	\end{equation}
	exist, where $\alpha=\alpha(a,b)$ is the number in~\eqref{eqq01alphadef}.
\end{theorem}
\begin{proof}
Let $r_n=|\FD^{(a,b)}_n|_1/|\FD^{(a,b)}_n|$ and $r'_n=|\FD^{(b,a)}_n|_1/|\FD^{(b,a)}_n|$.
	For $t\in[0,1]$ and $d\geqslant 1$ put
	\begin{equation}\label{eqq04fmap}
		f_d(t)=\frac{1-t}{d+1-dt}.
	\end{equation}
	The denominator in~\eqref{eqq04fmap} is at least $1$ on $[0,1]$, hence $f_d$ is
	continuous there.
	
	Let $x_n=\FD^{(b,a)}_n$. By~\eqref{eqq03dualfinite} we have
	$\FD^{(a,b)}_{n+1}=\sigma_a(x_n)$, and Lemma~\ref{lem01length} gives
	\begin{equation}\label{eqq04counts}
		|\FD^{(a,b)}_{n+1}|_1=|x_n|_0,\qquad |\FD^{(a,b)}_{n+1}|=a\,|x_n|_0+|x_n|.
	\end{equation}
	Dividing the two relations in~\eqref{eqq04counts} and writing $z_n=|x_n|_0/|x_n|=1-r'_n$
	gives
	\begin{equation}\label{eqq04recursion}
		r_{n+1}=\frac{z_n}{a z_n+1}=f_a\bigl(r'_n\bigr).
	\end{equation}
	The same argument applied to the pair $(b,a)$ gives $r'_{n+1}=f_b(r_n)$.
	
By Corollary~\ref{cor01sturmian} the words $\FD^{(a,b)}$ and $\FD^{(b,a)}$ are Sturmian. 

By~\eqref{eqq02slope} the frequency of the letter $1$ along the prefixes of a Sturmian word converges to its slope. Since $\FD^{(a,b)}_n$ is a prefix of $\FD^{(a,b)}$ and $|\FD^{(a,b)}_n|\to\infty$ by Lemma~\ref{lem03prefix}, the
ratio $r_n$ is a subsequence of that convergent sequence, hence $r_n\to x$, and likewise $r'_n\to y$, for some $x,y\in[0,1]$. Passing to the limit in~\eqref{eqq04recursion} and in its
	counterpart gives
	\begin{equation}\label{eqq04system}
		x=f_a(y),\qquad y=f_b(x).
	\end{equation}
	Hence $x=g(x)$ for $g=f_a\circ f_b$. A direct computation gives
	\begin{equation}\label{eqq04gmap}
		g(x)=\frac{b-(b-1)x}{ab+b+1-(ab+b-a)x},
	\end{equation}
	where the denominator is at least $a+1$ on $[0,1]$. Clearing it in
	$x=g(x)$ gives
	\begin{equation}\label{eqq04quadratic}
		(ab+b-a)\,x^2-(ab+2b)\,x+b=0.
	\end{equation}
	The leading coefficient in~\eqref{eqq04quadratic} is positive since $ab\geqslant a$
	and $b\geqslant 1$. Put $x_1=b/(\alpha+b)\in(0,1)$. Multiplying the left side
	of~\eqref{eqq04quadratic} at $x_1$ by $(\alpha+b)^2/b$ gives
	\begin{equation}\label{eqq04check}
		(ab+b-a)b-(ab+2b)(\alpha+b)+(\alpha+b)^2=\alpha^2-ab\,\alpha-ab=0,
	\end{equation}
	where the last equality is the defining equation of $\alpha$ from
	Lemma~\ref{lem01alphaquadratic}. Thus $x_1$ is a root
	of~\eqref{eqq04quadratic}. The product of the two roots is $b/(ab+b-a)$, hence
	the other root equals
	\begin{equation}\label{eqq04otherroot}
		x_2=\frac{\alpha+b}{ab+b-a}>1,
	\end{equation}
	since $\alpha>ab>ab-a$. The only root of~\eqref{eqq04quadratic} in $[0,1]$ is
	therefore $x_1$, which gives $x=b/(\alpha+b)$ and $\operatorname{freq}_0=1-x=\alpha/(\alpha+b)$.
\end{proof}

Theorem~\ref{thm02freqduality} agrees with~\eqref{eqq03slopealpha}. The
derivation above uses only $\sigma_a$ and the identity~\eqref{eqq03dualinfinite},
and it makes visible where the dependence on the order of $a$ and $b$ comes
from, namely from the asymmetry of $\sigma_a$ in the two letters.

\begin{corollary}\label{cor04freqasym}
	For all $a,b\geqslant 1$,
	\begin{equation}\label{eqq04asym}
		\operatorname{freq}_1\bigl(\FD^{(a,b)}\bigr)=\operatorname{freq}_1\bigl(\FD^{(b,a)}\bigr)
	\end{equation}
	if and only if $a=b$.
\end{corollary}
\begin{proof}
	By Theorem~\ref{thm02freqduality} and $\alpha(a,b)=\alpha(b,a)$, equality
	in~\eqref{eqq04asym} means $b/(\alpha+b)=a/(\alpha+a)$. Cross multiplication
	gives $b\alpha+ab=a\alpha+ab$, that is $\alpha(b-a)=0$, and $\alpha>0$.
\end{proof}

For $a=b=1$ we have $\alpha=\fphi$ and the frequency of the letter $1$ is
$1/(\fphi+1)=1/\fphi^{2}$, which is the classical value. For $(a,b)=(1,2)$ we
have $\alpha=1+\sqrt3\approx 2.7321$ and~\eqref{eqq04freqs} gives
$\operatorname{freq}_1\bigl(\FD^{(1,2)}\bigr)\approx 0.4226$. The word
$\FD^{(1,2)}_9$ has length $265$ and contains $112$ letters $1$, and
$112/265\approx 0.4226$. For $(a,b)=(2,1)$ the formula gives
$\operatorname{freq}_1\bigl(\FD^{(2,1)}\bigr)=1/(\alpha+1)\approx 0.2679$, and the
word $\FD^{(2,1)}_9$ has length $209$ and contains $56$ letters $1$, and
$56/209\approx 0.2679$.

%=============================
\section{Return Words and the Derived Sequence}\label{sec5duality}

Definition~\ref{def01returnword} and Theorem~\ref{thm01vuillon} show that every
factor of a Sturmian word has exactly two return words. This subsection
determines them for the letters $0$ and $1$ in $\FD^{(a,b)}$, and describes the
sequence obtained by coding $\FD^{(a,b)}$ with the return words of $0$. In both
cases the block decomposition $\FD^{(a,b)}=\sigma_a(\FD^{(b,a)})$ from
Theorem~\ref{thm01desub} gives the answer, and the dependence on the order of
$a$ and $b$ is visible directly in the result.

The first lemma records two local properties of $\FD^{(a,b)}$ and of its image
under the duality.

\begin{lemma}\label{lem05local}
	For all $a,b\geqslant 1$ the word $\FD^{(a,b)}$ satisfies the following.
	\begin{enumerate}
		\item It begins with $0$, and every occurrence of $1$ is followed by $0$.
		\item The words $00$ and $010$ are factors of $\FD^{(a,b)}$.
	\end{enumerate}
\end{lemma}
\begin{proof}
	By~\eqref{eqq03dualinfinite} the word $\FD^{(a,b)}$ is a concatenation of the
	blocks $\sigma_a(0)=0^a1$ and $\sigma_a(1)=0$. Every block begins with $0$, and
	a letter $1$ occurs only as the last letter of a block $0^a1$. This gives the
	first part.
	
	The word $\FD^{(a,b)}$ begins with the block $0^a1$, followed by a block that
	begins with $0$. Hence $010$ is a factor.
	
	By Theorem~\ref{thm02freqduality} the frequency of the letter $0$ is
	\begin{equation}\label{eqq05freqzero}
		\operatorname{freq}_0\bigl(\FD^{(a,b)}\bigr)=\frac{\alpha}{\alpha+b}>\frac12,
	\end{equation}
	since $\alpha>ab\geqslant b$ by~\eqref{eqq01alphadef}. If $00$ were not a factor,
	every letter $0$ would be followed by $1$, and every prefix of length $n$ would
	contain at most $(n+1)/2$ letters $0$. This contradicts~\eqref{eqq05freqzero}.
\end{proof}

The return words of the letter $0$ do not depend on $a$ or $b$.

\begin{theorem}\label{thm01return0}
	For all $a,b\geqslant 1$ the return words of $0$ in $\FD^{(a,b)}$ are exactly $0$
	and $01$.
\end{theorem}
\begin{proof}
	The letter $0$ occurs at position $0$. Let $0$ occur at a position $p$. If the
	letter at position $p+1$ is $0$ then the return word starting at $p$ is $0$. If it
	is $1$ then the letter at position $p+2$ is $0$ by Lemma~\ref{lem05local}, and the
	return word is $01$. Hence every return word of $0$ belongs to $\{0,01\}$.
	Both occur, since $00$ and $010$ are factors by Lemma~\ref{lem05local}.
\end{proof}

The return words of the letter $1$ carry the dependence on $a$.

\begin{theorem}\label{thm02return1}
	For all $a,b\geqslant 1$ the return words of $1$ in $\FD^{(a,b)}$ are exactly
	$10^{a}$ and $10^{a+1}$.
\end{theorem}
\begin{proof}
	Write $G=\FD^{(b,a)}$, so that $\FD^{(a,b)}=\sigma_a(G)$. Each letter $1$ of
	$\FD^{(a,b)}$ is the last letter of a block $\sigma_a(0)=0^a1$, and the
	occurrences of $1$ in $\FD^{(a,b)}$ correspond in order to the occurrences of $0$
	in $G$.
	
	Consider two consecutive occurrences of $0$ in $G$, separated by exactly $k$
	letters $1$. By Lemma~\ref{lem05local} applied to the pair $(b,a)$ the word $G$
	has no factor $11$, hence $k\in\{0,1\}$. The image of the segment of $G$ between
	the two occurrences is $\sigma_a(1)^k\sigma_a(0)=0^k\,0^a1$. Hence the return word
	of $1$ between the corresponding occurrences of $1$ in $\FD^{(a,b)}$ is
	\begin{equation}\label{eqq05return1}
		1\,0^{a+k},\qquad k\in\{0,1\}.
	\end{equation}
	The value $k=0$ occurs exactly when $00$ is a factor of $G$, and $k=1$ occurs
	exactly when $010$ is a factor of $G$. Both hold by Lemma~\ref{lem05local}
	applied to $G$.
\end{proof}

The next example checks Theorems~\ref{thm01return0} and~\ref{thm02return1} on
two short words.

\begin{example}\label{ex05return}
	For $a=b=1$ the Fibonacci word begins with $0100101001001010$. The letter $1$
	occurs at the positions $1,4,6,9,12,14$, and the corresponding return words are
	$100$, $10$, $100$, $100$, $10$. This agrees with the set $\{10,100\}$ given by
	Theorem~\ref{thm02return1}.
	
	For $(a,b)=(2,1)$ the recurrence~\eqref{eqq01biperiodicrecurrence} gives
	\begin{equation}\label{eqq05example}
		\FD^{(2,1)}_5=001000100010010.
	\end{equation}
	The letter $1$ occurs in~\eqref{eqq05example} at the positions $2,6,10,13$, with
	return words $1000$, $1000$, $100$. This agrees with $\{10^2,10^3\}$.
\end{example}

The theorems above compare the two words $\FD^{(a,b)}$ and $\FD^{(b,a)}$ as
follows.

\begin{corollary}\label{cor01returnduality}
	The return words of $1$ in $\FD^{(a,b)}$ and in $\FD^{(b,a)}$ are
	$\{10^a,\,10^{a+1}\}$ and $\{10^b,\,10^{b+1}\}$ respectively, and the return
	words of $0$ in both words are $\{0,01\}$. In particular the sets of return words
	of $1$ coincide if and only if $a=b$.
\end{corollary}
\begin{proof}
	This is Theorem~\ref{thm02return1} applied to the pairs $(a,b)$ and $(b,a)$,
	together with Theorem~\ref{thm01return0}. The two sets
	$\{10^a,10^{a+1}\}$ and $\{10^b,10^{b+1}\}$ are equal exactly when $a=b$.
\end{proof}

Coding $\FD^{(a,b)}$ by the return words of $0$ gives a more precise form of the
duality. Let $A=0$ and $B=01$ denote the two return words of $0$, and let the
derived sequence $D^{(a,b)}$ be the infinite word over the alphabet $\{0,1\}$
obtained by writing $\FD^{(a,b)}$ as a concatenation of copies of $A$ and $B$
and replacing $A$ by $0$ and $B$ by $1$. The formula $\sigma_d(0)=0^d1$,
$\sigma_d(1)=0$ is also used for $d=0$, in which case $\sigma_0$ exchanges the two
letters.

\begin{theorem}\label{thm05derived}
	For all $a,b\geqslant 1$,
	\begin{equation}\label{eqq05derived}
		D^{(a,b)}=\sigma_{a-1}\bigl(\FD^{(b,a)}\bigr).
	\end{equation}
\end{theorem}
\begin{proof}
Let $G=\FD^{(b,a)}$. The block $\sigma_a(0)=0^a1$ equals $A^{a-1}B$ as a word,
	and the block $\sigma_a(1)=0$ equals $A$. Hence $\FD^{(a,b)}=\sigma_a(G)$ is
	written as a concatenation of the words $A$ and $B$ by replacing each letter $0$
	of $G$ with $A^{a-1}B$ and each letter $1$ of $G$ with $A$. It remains to check
	that this decomposition is the decomposition into return words of $0$. A letter
	$0$ inside a block $0^a1$ that is not the last letter $0$ of the block is followed
	by $0$, and its return word is $A$. The last letter $0$ of a block $0^a1$ is
	followed by $1$, and its return word is $B$. The letter $0$ of a block $0$ is
	followed by the first letter of the next block, which is $0$, and its return word
	is $A$. The resulting coding replaces $0$ by $0^{a-1}1$ and $1$ by $0$, which is
	$\sigma_{a-1}$.
\end{proof}

For $a=1$ the identity~\eqref{eqq05derived} says that $D^{(1,b)}$ is the word $\FD^{(b,1)}$ with its letters exchanged. 
For $a=b=1$ this is the classical fact that the derived sequence of the Fibonacci word with respect to the letter $0$ is its complement. 
For $(a,b)=(2,1)$ the first return words of $0$ in
\eqref{eqq05example} are $A,B,A,A,B$, which is the prefix $01001$ of $\sigma_1(\FD^{(1,2)})$ computed from $\FD^{(1,2)}=01010\cdots$. 

%=================================
\section{Palindromic Prefixes and the Duality}\label{sec01palduality}

The standard words $s_k$ of Section~\ref{sec01sadic} are close to palindromes, and this property passes through the morphism $\sigma_a$. 
This subsection proves that the word obtained from $\FD^{(a,b)}_n$ by removing its last two letters is a palindrome, and then shows that the duality \eqref{eqq03dualinfinite} transports palindromic prefixes of $\FD^{(b,a)}$ to palindromic prefixes of $\FD^{(a,b)}$. For $m\geqslant 1$ let $w_m=s_m^{-}$ in the notation of~\eqref{eqq02minus}. 
This is defined since $L_m\geqslant 2$ for $m\geqslant 1$.

\begin{lemma}\label{lem04commute}
	For every $n\geqslant 0$ the words $s_ns_{n-1}$ and $s_{n-1}s_n$ have a common
	prefix $p_n$ of length $L_n+L_{n-1}-2$, and they differ in their last two
	letters.
\end{lemma}
\begin{proof}
	For $n=0$ we have $s_0s_{-1}=01$ and $s_{-1}s_0=10$, and $p_0=\varepsilon$. Assume
	the claim for $n$, and let $d=d_{n+1}$. By~\eqref{eqq03standard},
	\begin{equation}\label{eqq04commutea}
		s_{n+1}s_n=s_n^{\,d}\,s_{n-1}s_n,\qquad s_ns_{n+1}=s_n\,s_n^{\,d}\,s_{n-1}=s_n^{\,d}\,s_ns_{n-1}.
	\end{equation}
	By the induction hypothesis, $s_{n-1}s_n$ and $s_ns_{n-1}$ share the prefix
	$p_n$ and differ in the last two letters. Prefixing both by $s_n^{\,d}$ shows that
	$s_{n+1}s_n$ and $s_ns_{n+1}$ share the prefix $p_{n+1}=s_n^{\,d}p_n$ and differ
	in the last two letters. Its length is $dL_n+L_n+L_{n-1}-2=L_{n+1}+L_n-2$, by
	\eqref{eqq03lengths}.
\end{proof}

\begin{theorem}\label{thm01centralpalindrome}
For all $a,b\geqslant 1$ and all $n\geqslant 2$ the word $\bigl(\FD^{(a,b)}_n\bigr)^{-}$ is a palindrome.
\end{theorem}
\begin{proof}
By Proposition~\ref{prop01fabisstandard} we have $\FD^{(a,b)}_n=s_{n-1}$.
It suffices to show that $w_m$ is a palindrome for every $m\geqslant 1$, and this is established by strong induction on $m$.

For $m=1$ the word $s_1=0^a1$ gives $w_1=0^{a-1}$, which is a palindrome. For $m=2$ the word $s_2=(0^a1)^{d_2}0$ gives
\begin{equation}\label{eqq06wtwo}
w_2=(0^a1)^{d_2-1}0^a,
\end{equation}
and the reversal of~\eqref{eqq06wtwo} is $0^a(10^a)^{d_2-1}$, which is the same 	word.
	
Let $m=k+1$ with $k\geqslant 2$, and assume that $w_{k-1}$ and $w_k$ are palindromes. Since $L_k\geqslant 2$ and $L_{k-1}\geqslant 2$,
\begin{equation}\label{eqq06decomp}
s_k=w_k\,xy,\qquad s_{k-1}=w_{k-1}\,x'y',
\end{equation}
for letters $x,y,x',y'$. Lemma~\ref{lem04commute} applied with $n=k$ shows that $s_ks_{k-1}$ and $s_{k-1}s_k$ have the common prefix of length $L_k+L_{k-1}-2$, which by~\eqref{eqq06decomp} gives
\begin{equation}\label{eqq06prefix}
w_k\,xy\,w_{k-1}=w_{k-1}\,x'y'\,w_k,
\end{equation}
and that the two words end with $x'y'$ and $xy$ respectively, with 	$x'y'\neq xy$. 
Both words contain the same number of each letter and they agree before their last two positions. 
Hence $x'y'$ is obtained from $xy$ by exchanging the two letters, and $x\neq y$. 
Thus $x'y'=yx$, and \eqref{eqq06prefix} satisfies 
\begin{equation}\label{eqq06rel}
w_k\,xy\,w_{k-1}=w_{k-1}\,yx\,w_k.
\end{equation}
Let $U=w_kxy$ and $V=yx\,w_k$. Since $w_k$ is a palindrome, $V$ is the reversal of $U$.
The relation~\eqref{eqq06rel} reads $Uw_{k-1}=w_{k-1}V$, and induction on $j$ gives
\begin{equation}\label{eqq06power}
U^jw_{k-1}=w_{k-1}V^j\qquad\text{for all } j\geqslant 0,
\end{equation}
since $U^{j+1}w_{k-1}=U\,w_{k-1}V^j=w_{k-1}V^{j+1}$. 
Since $s_{k+1}=s_k^{\,d_{k+1}}s_{k-1}$ and $|s_{k-1}|\geqslant 2$, removing the last two letters gives $w_{k+1}=U^{d_{k+1}}w_{k-1}$. 
Its reversal is $w_{k-1}V^{d_{k+1}}$, and this equals $w_{k+1}$ by~\eqref{eqq06power}. Hence $w_{k+1}$ is a palindrome, which completes the induction.
\end{proof}

The theorem gives palindromic prefixes of $\FD^{(a,b)}$ of explicit lengths.

\begin{corollary}\label{cor002arbitrary}
For all $a,b\geqslant 1$ and all $n\geqslant 2$ the word $\bigl(\FD^{(a,b)}_n\bigr)^{-}$ is a palindromic prefix of $\FD^{(a,b)}$ of length
\begin{equation}\label{eqq06length}
\bigl|\bigl(\FD^{(a,b)}_n\bigr)^{-}\bigr|=L_{n-1}-2,
\end{equation}
where $L_k$ satisfies~\eqref{eqq03lengths}. In particular $\FD^{(a,b)}$ has palindromic prefixes of arbitrarily large length.
\end{corollary}
\begin{proof}
Since $\FD^{(a,b)}_n=s_{n-1}$ is a prefix of $\FD^{(a,b)}$ and by
Lemma~\ref{lem03prefix}, the same holds for its prefix
$\bigl(\FD^{(a,b)}_n\bigr)^{-}$. This word is a palindrome by
Theorem~\ref{thm01centralpalindrome}. Its length is given by~\eqref{eqq06length}, and this length tends to infinity by Lemma~\ref{lem03prefix}.
\end{proof}

By~\eqref{eqq03palindromic} every length admits one or two palindromic factors.
Corollary~\ref{cor002arbitrary} adds that for the lengths in~\eqref{eqq06length}
a palindromic factor can be chosen as a prefix.

The second result transports palindromes through the duality. Let
$\sigma'_a$ be the morphism
\begin{equation}\label{eqq06sigmaprime}
\sigma'_a(0)=10^a,\qquad \sigma'_a(1)=0.
\end{equation}

\begin{lemma}\label{lem06conj}
For every finite word $x$ and every $a\geqslant 1$,
\begin{equation}\label{eqq06conj}
\sigma_a(x)\,0^a=0^a\,\sigma'_a(x),\qquad \widetilde{\sigma_a(x)}=\sigma'_a(\widetilde{x}).
\end{equation}
\end{lemma}
\begin{proof}
For the letter $0$ we have $\sigma_a(0)0^a=0^a10^a=0^a\sigma'_a(0)$, and for the letter $1$ we have $\sigma_a(1)0^a=0^{a+1}=0^a\sigma'_a(1)$. If the first identity 	holds for $x$ and for a letter $c$, then
\begin{equation}\label{eqq06induct}
\sigma_a(xc)0^a=\sigma_a(x)\,0^a\,\sigma'_a(c)=0^a\,\sigma'_a(x)\sigma'_a(c)=0^a\sigma'_a(xc).
\end{equation}
For the second identity, the reversal of $\sigma_a(0)=0^a1$ is $10^a=\sigma'_a(0)$, and the reversal of $\sigma_a(1)=0$ is $0=\sigma'_a(1)$. 
The reversal of a product of words is the product of the reversals in the opposite order, which gives the claim for every $x$.
\end{proof}

\begin{theorem}\label{thm06dualpal}
Let $x$ be a palindromic prefix of $\FD^{(b,a)}$. Then $\sigma_a(x)0^a$ is a palindromic prefix of $\FD^{(a,b)}$ of length
\begin{equation}\label{eqq06dlength}
\bigl|\sigma_a(x)0^a\bigr|=a|x|_0+|x|+a.
\end{equation}
\end{theorem}
\begin{proof}
The length~\eqref{eqq06dlength} is~\eqref{eqq04length} with the $a$ added letters $0$. By Lemma~\ref{lem06conj} and $\widetilde x=x$,
\begin{equation}\label{eqq06palcheck}
\widetilde{\sigma_a(x)0^a}=0^a\,\widetilde{\sigma_a(x)}=0^a\,\sigma'_a(x)=\sigma_a(x)0^a,
\end{equation}
and hence $\sigma_a(x)0^a$ is a palindrome.

Let $G=\FD^{(b,a)}$ and let $c$ be the letter of $G$ that follows $x$. If $c=0$,
then $\sigma_a(xc)=\sigma_a(x)0^a1$. If $c=1$, then the next letter of $G$ is $0$,
since every occurrence of $1$ in $G$ is followed by $0$ by
Lemma~\ref{lem05local} applied to the pair $(b,a)$, and
$\sigma_a(x10)=\sigma_a(x)\,0\,0^a1$. In both cases $\sigma_a(x)0^a$ is a prefix
of $\sigma_a(G)$, which equals $\FD^{(a,b)}$ by~\eqref{eqq03dualinfinite}.
\end{proof}

\begin{example}\label{ex06pal}
For $a=b=1$ the word $x=010$ is a palindromic prefix of the Fibonacci word $0100101\cdots$. 
Then $\sigma_1(010)0=01001\,0=010010$, which is a palindromic prefix of the Fibonacci word.
	
For $(a,b)=(2,1)$ the word $\FD^{(1,2)}$ begins with $0101001$, and $x=010$ is a palindromic prefix. 
Then $\sigma_2(010)00=001\,0\,001\,00=001000100$, which is a palindrome. 
It is a prefix of $\FD^{(2,1)}=0010001000100\cdots$, as stated by Theorem~\ref{thm06dualpal}.
\end{example}

%==============================
%========================
\section{Continued Fractions, the Slope and the Critical Exponent}\label{sec01continuedfraction}

This section determines the continued fraction expansions attached to $\FD^{(a,b)}$ and derives the critical exponent from Theorem~\ref{thm01damaniklenz}. 
The number $\alpha$ of~\eqref{eqq01alphadef} is expanded first. 
The slope is treated next, and the critical exponent is then
obtained as the limit superior of two explicit sequences of finite continued fractions.
A continued fraction $[c_0;c_1,c_2,\dots]$ with positive integer entries denotes the limit of its convergents, and $[c_0;\overline{c_1,\dots,c_k}]$ denotes the expansion that repeats the block $c_1,\dots,c_k$ indefinitely. 
The number $\chi(a,b)=[a;b,a,b,\dots]$ of~\eqref{eqq03chi} is the purely periodic
case.

\begin{lemma}\label{lem01alphaquadratic}
For all $a,b\geqslant 1$ the number $\alpha=\alpha(a,b)$ is the unique positive
root of
\begin{equation}\label{eqq05minpoly}
	t^2-ab\,t-ab=0.
\end{equation}
	The other root $\alpha'$ satisfies
	\begin{equation}\label{eqq05vieta}
		\alpha+\alpha'=ab,\qquad \alpha\alpha'=-ab,
	\end{equation}
	and $\alpha>ab$. Moreover $\alpha$ is irrational.
\end{lemma}
\begin{proof}
	The discriminant of~\eqref{eqq05minpoly} is $D=a^2b^2+4ab$, and the roots are
	$(ab\pm\sqrt D)/2$. Their product is $-ab<0$, hence exactly one root is
	positive, and it is $\alpha$ by~\eqref{eqq01alphadef}. The relations
	in~\eqref{eqq05vieta} are the usual relations between roots and coefficients.
	Since $D>a^2b^2$ we have $\sqrt D>ab$ and $\alpha>ab$.
	
	Since $D<(ab+2)^2$ the number $\sqrt D$ lies strictly between $ab$ and $ab+2$.
	If $\sqrt D$ were rational it would be an integer, hence equal to $ab+1$. The
	equation $D=(ab+1)^2$ reads $2ab=1$, which has no solution in integers.
	Consequently $\sqrt D$ is irrational, and so is $\alpha$.
\end{proof}

The next theorem identifies the continued fraction of $\alpha$. It depends on
$a$ and $b$ only through the product $ab$.

\begin{theorem}\label{thm01alphacf}
	For all $a,b\geqslant 1$,
	\begin{equation}\label{eqq05alphacf}
		\alpha(a,b)=[\overline{ab,\,1}].
	\end{equation}
\end{theorem}
\begin{proof}
	Let $x=[\overline{ab,1}]$. Then $x>ab\geqslant 1$, and periodicity gives
	\begin{equation}\label{eqq05fixed}
		x=ab+\cfrac{1}{1+\cfrac{1}{x}}=ab+\frac{x}{x+1}.
	\end{equation}
	Clearing denominators in~\eqref{eqq05fixed} gives $x^2+x=abx+ab+x$, that is,
	$x^2-ab\,x-ab=0$. This is~\eqref{eqq05minpoly}, and $x>0$. By
	Lemma~\ref{lem01alphaquadratic} the number $x$ equals $\alpha$.
\end{proof}

\begin{example}\label{ex05alpha}
	For $a=b=1$ the expansion~\eqref{eqq05alphacf} is $[\overline{1,1}]=\fphi$. For
	$(a,b)=(2,3)$ it is $[\overline{6,1}]$, and numerically
	$\alpha=(6+\sqrt{60})/2\approx 6.8730$.
\end{example}

The slope of $\FD^{(a,b)}$ is not a function of $\alpha$ alone. The next theorem
gives its value and its expansion, and checks the expansion of Corollary~\ref{cor01sturmian} independently of the argument in Section~\ref{sec01sadic}.

\begin{theorem}\label{thm01slopecf}
For all $a,b\geqslant 1$ the slope of $\FD^{(a,b)}$ is
\begin{equation}\label{eqq05slopecf}
\theta^{(a,b)}=\frac{b}{\alpha+b}=\bigl[\,0;\,a+1,\,\overline{b,\,a}\,\bigr].
\end{equation}
\end{theorem}
\begin{proof}
	Let $y=[0;a+1,\overline{b,a}]$. The number $[\overline{b,a}]$ equals
	$\chi(b,a)$, which is $\alpha/a$ by~\eqref{eqq03chialpha}. Hence
	\begin{equation}\label{eqq05ycompute}
		y=\cfrac{1}{a+1+\cfrac{1}{\alpha/a}}=\frac{\alpha}{(a+1)\alpha+a}.
	\end{equation}
	The equality $y=b/(\alpha+b)$ is equivalent to
	$\alpha(\alpha+b)=b\bigl((a+1)\alpha+a\bigr)$, that is, to
	$\alpha^2=ab\,\alpha+ab$, which is~\eqref{eqq05minpoly}. By
	Theorem~\ref{thm02freqduality} and~\eqref{eqq02slope} the slope of
	$\FD^{(a,b)}$ equals $b/(\alpha+b)$.
\end{proof}

\begin{example}\label{ex05slope}
For $a=b=1$ the expansion~\eqref{eqq05slopecf} is $[0;2,\overline{1}]$, which is the expansion of $1/\fphi^{2}\approx 0.3820$. 
For $(a,b)=(1,2)$ it is $[0;2,\overline{2,1}]\approx 0.4226$, and for $(a,b)=(2,1)$ it is	$[0;3,\overline{1,2}]\approx 0.2679$.
\end{example}

The number $\alpha$ depends on $(a,b)$ only through $ab$, but the slope does not.
The pairs $(2,3)$ and $(1,6)$ have $ab=6$ and the same value
$\alpha\approx 6.8730$. 
By~\eqref{eqq05slopecf} their slopes are $3/(\alpha+3)\approx 0.3039$ and $6/(\alpha+6)\approx 0.4661$, which are different. 
The two words differ as well in the return words of the letter $1$ by Theorem~\ref{thm02return1}.
The critical exponent is now computed from Theorem~\ref{thm01damaniklenz}. 
By \eqref{eqq05slopecf} the partial quotients of the slope are
\begin{equation}\label{eqq05digits}
	a_1=a+1,\qquad a_n=\begin{cases}a & n\text{ odd},\ n\geqslant 3,\\ b & n\text{ even}.\end{cases}
\end{equation}
For $n\geqslant 1$ denote
\begin{equation}\label{eqq05Rn}
	R_n=[a_{n+1};\,a_n,\,a_{n-1},\dots,a_1],
\end{equation}
a finite continued fraction read from index $n+1$ down to index $1$.

\begin{lemma}\label{lem05limits}
For all $a,b\geqslant 1$, the subsequences of $(R_n)$ with even and odd index converge, and
\begin{equation}\label{eqq05limits}
\lim_{m\to\infty} R_{2m}=\chi(a,b)=\frac{\alpha}{b},
\qquad
\lim_{m\to\infty} R_{2m+1}=\chi(b,a)=\frac{\alpha}{a}.
\end{equation}
In particular the set of limit points of $(R_n)$ is $\{\alpha/a,\alpha/b\}$.
\end{lemma}
\begin{proof}
Let $n=2m$ be even. 
By~\eqref{eqq05digits} we have $a_{n+1}=a$, and the entries alternate between $a$ and $b$ as the index decreases to $2$, where the entry is $b$. 
The last entry is $a_1=a+1$. 
Denote by $c_0,c_1,c_2,\dots$ the partial quotients of $\chi(a,b)$. Then, $c_j=a$ for even $j$ and $c_j=b$ for odd $j$, and denote by $p_j/q_j$ its convergents. Then
\begin{equation}\label{eqq05Rnform}
R_n=[c_0;c_1,\dots,c_{n-1},a+1]=\frac{(a+1)\,p_{n-1}+p_{n-2}}{(a+1)\,q_{n-1}+q_{n-2}}.
\end{equation}
The number $\chi(a,b)$ has the same representation~\eqref{eqq05Rnform} with $a+1$ replaced by the complete quotient $t=[c_n;c_{n+1},\dots]$, and $c_n=a$ gives $a<t<a+1$. 
For positive reals $t,t'$ the identity 	$p_{n-1}q_{n-2}-p_{n-2}q_{n-1}=\pm 1$ gives
\begin{equation}\label{eqq05diff}
\left|\frac{t\,p_{n-1}+p_{n-2}}{t\,q_{n-1}+q_{n-2}}-\frac{t'\,p_{n-1}+p_{n-2}}{t'\,q_{n-1}+q_{n-2}}\right|
=\frac{|t-t'|}{(t\,q_{n-1}+q_{n-2})(t'\,q_{n-1}+q_{n-2})}\leqslant\frac{|t-t'|}{q_{n-1}^2}.
\end{equation}
With $t'=a+1$ and the value of $t$ above, the right side of~\eqref{eqq05diff} is at most $q_{n-1}^{-2}$, which tends to $0$ because $q_j\geqslant q_{j-1}+q_{j-2}$.
This gives the first limit in~\eqref{eqq05limits}.
For odd $n=2m+1$ the same argument applies to $\chi(b,a)$: now $a_{n+1}=b$, the entries alternate down to $a_2=b$, and the last entry $a_1=a+1$ replaces the quotient $c_n=a$ of $\chi(b,a)=[c_0;c_1,c_2,\dots]$.
\end{proof}
It is known that $2+\fphi$ is the smallest critical exponent among all
Sturmian words~\cite{MignosiPirillo1992, DamanikLenz2002}. The next result
shows that, within the biperiodic family, this value is attained only by the
Fibonacci word.
\begin{theorem}\label{thm05ce}
	For all $a,b\geqslant 1$,
	\begin{equation}\label{eqq05ce}
		\CE\bigl(\FD^{(a,b)}\bigr)=2+\frac{\alpha}{\min(a,b)}.
	\end{equation}
\end{theorem}
\begin{proof}
	By Corollary~\ref{cor01sturmian} the word $\FD^{(a,b)}$ is Sturmian with the slope~\eqref{eqq05slopecf}, and Theorem~\ref{thm01damaniklenz} gives $\CE\bigl(\FD^{(a,b)}\bigr)=2+\limsup_{n\to\infty}R_n$. 
	By Lemma~\ref{lem05limits} the set of limit points of $(R_n)$ is
	$\{\alpha/a,\ \alpha/b\}$, so $\limsup_{n\to\infty}R_n=\max\{\alpha/a,\alpha/b\}$.
	Since $\alpha>0$ and $t\mapsto\alpha/t$ is decreasing on $(0,\infty)$, this
	maximum equals $\alpha/\min(a,b)$.
\end{proof}

\begin{corollary}\label{cor05symmetry}
	For all $a,b\geqslant 1$,
	\begin{equation}\label{eqq05symmetry}
		\CE\bigl(\FD^{(a,b)}\bigr)=\CE\bigl(\FD^{(b,a)}\bigr).
	\end{equation}
\end{corollary}
\begin{proof}
	The right side of~\eqref{eqq05ce} is invariant under the exchange of $a$ and $b$,
	since $\alpha(a,b)=\alpha(b,a)$ and $\min(a,b)=\min(b,a)$.
\end{proof}

Corollary~\ref{cor05symmetry} contrasts with Corollary~\ref{cor04freqasym}. The
critical exponent is invariant under the duality~\eqref{eqq03dualinfinite} while
the frequency of the letter $1$ is not. The critical exponent still depends on
more than the product $ab$. The pairs $(2,3)$ and $(1,6)$ give
$2+\alpha/2\approx 5.4365$ and $2+\alpha\approx 8.8730$ respectively.

It is known that $2+\fphi$ is the smallest critical exponent among all
Sturmian words~\cite{MignosiPirillo1992, DamanikLenz2002}. The next result
shows that, within the biperiodic family, this value is attained only by the
Fibonacci word.

\begin{proposition}\label{thm01globalmin}
	For all $a,b\geqslant 1$,
	\begin{equation}\label{eqq05globalmin}
		\CE\bigl(\FD^{(a,b)}\bigr)\geqslant 2+\fphi,
	\end{equation}
	with equality if and only if $a=b=1$.
\end{proposition}
\begin{proof}
	Let $m=\min(a,b)$, $M=\max(a,b)$ and $x=\alpha/m$. If $a\geqslant b$ then $m=b$ and
	$x=\chi(a,b)$ by~\eqref{eqq03chialpha}. If $a<b$ then $m=a$ and $x=\chi(b,a)$. In
	both cases $x=\chi(M,m)$, and~\eqref{eqq03quadx} with the pair $(M,m)$ shows that
	$x$ is the positive root of
	\begin{equation}\label{eqq05hpoly}
		h(t)=m\,t^2-Mm\,t-M.
	\end{equation}
	The polynomial $h$ is convex with $h(0)=-M<0$, hence $h(t)\leqslant 0$ exactly for
	$0\leqslant t\leqslant x$. Using $\fphi^2=\fphi+1$ we obtain
	\begin{equation}\label{eqq05hphi}
		h(\fphi)=m(\fphi+1)-Mm\fphi-M=(m-M)-m\fphi\,(M-1).
	\end{equation}
	Both terms on the right side of~\eqref{eqq05hphi} are at most $0$, since
	$m\leqslant M$ and $M\geqslant 1$. Hence $h(\fphi)\leqslant 0$ and $x\geqslant\fphi$. By
	Theorem~\ref{thm05ce} this gives~\eqref{eqq05globalmin}. Equality requires
	$h(\fphi)=0$, which forces $m=M$ and $M=1$, that is, $a=b=1$.
\end{proof}
The values of the main quantities for five pairs are collected in
Table~\ref{tab05values}.

\begin{table}[H]
	\centering
	\begin{tabular}{ccccc}
		\hline
		$(a,b)$ & $\alpha$ & $\operatorname{freq}_1$ & $\theta^{(a,b)}$ & $\CE$ \\
		\hline
		$(1,1)$ & $1.6180$ & $0.3820$ & $[0;2,\overline{1}]$ & $3.6180$ \\
		$(1,2)$ & $2.7321$ & $0.4226$ & $[0;2,\overline{2,1}]$ & $4.7321$ \\
		$(2,1)$ & $2.7321$ & $0.2679$ & $[0;3,\overline{1,2}]$ & $4.7321$ \\
		$(2,3)$ & $6.8730$ & $0.3039$ & $[0;3,\overline{3,2}]$ & $5.4365$ \\
		$(1,6)$ & $6.8730$ & $0.4661$ & $[0;2,\overline{6,1}]$ & $8.8730$ \\
		\hline
	\end{tabular}
	\caption{Values of $\alpha$, of the frequency of the letter $1$, of the slope and
		of the critical exponent.}\label{tab05values}
\end{table}
%%%%%%%%%%%%%%%%%%%%%%%%%%%%
\section{Conclusion}\label{sec06conclusion}

The biperiodic Fibonacci word $\FD^{(a,b)}$ is the standard Sturmian word with
directive sequence $(a,b,a,b,\dots)$ by Proposition~\ref{prop01fabisstandard}.
The paper started from the observation that the number $\alpha(a,b)$ is
symmetric in $a$ and $b$ while several invariants of $\FD^{(a,b)}$ are not. The
morphism $\sigma_a$ explains the relation between the two words, since
Theorem~\ref{thm01desub} gives
\begin{equation}\label{eqq06final}
	\FD^{(a,b)}=\sigma_a\bigl(\FD^{(b,a)}\bigr)
\end{equation}
exactly, without relabelling of letters and without prefix correction.
Corollary~\ref{cor03uniquepreimage} shows that~\eqref{eqq06final} determines
$\FD^{(b,a)}$ uniquely.

The invariants split into two groups under the exchange of $a$ and $b$. The
first group consists of the quantities that depend on the order of the
parameters. The frequency of the letter $1$ equals $b/(\alpha+b)$ by
Theorem~\ref{thm02freqduality}, and by Corollary~\ref{cor04freqasym} it is
invariant under the exchange only for $a=b$. The return words of $1$ are
$10^a$ and $10^{a+1}$ by Theorem~\ref{thm02return1}, and
Corollary~\ref{cor01returnduality} shows that the exchange replaces the exponent
$a$ by $b$. The derived sequence with respect to the letter $0$ is
$\sigma_{a-1}(\FD^{(b,a)})$ by Theorem~\ref{thm05derived}.

The second group consists of the quantities that are invariant. The balance
function, the abelian complexity and the palindromic complexity are the same for
all $a,b$ by Corollary~\ref{cor03invariants}, and the return words of $0$ are
$0$ and $01$ by Theorem~\ref{thm01return0}. The critical exponent equals
\begin{equation}\label{eqq06cefinal}
	\CE\bigl(\FD^{(a,b)}\bigr)=2+\frac{\alpha}{\min(a,b)}
\end{equation}
by Theorem~\ref{thm05ce}, and Corollary~\ref{cor05symmetry} shows that it is
invariant under the exchange. By Proposition~\ref{thm01globalmin} its minimum
over all pairs is $2+\fphi$, attained only at $a=b=1$.

The palindromic structure is compatible with the duality. The word
$\bigl(\FD^{(a,b)}_n\bigr)^{-}$ is a palindrome for every $n\geqslant 2$ by
Theorem~\ref{thm01centralpalindrome}, and Theorem~\ref{thm06dualpal} transports
palindromic prefixes from $\FD^{(b,a)}$ to $\FD^{(a,b)}$ through $\sigma_a$. The
slope has the expansion~\eqref{eqq05slopecf}, and it depends on the pair $(a,b)$
and not only on the product $ab$, in contrast with $\alpha$ and the expansion
in~\eqref{eqq05alphacf}.

The comparison of~\eqref{eqq06final} and~\eqref{eqq06cefinal} shows that the
asymmetry of $\sigma_a$ in the two letters is visible in the letter frequencies
and in the return words of $1$, and it disappears in the critical exponent,
which is governed by the limit superior of two continued fractions exchanged by
the duality.
%%%%%%%%%%%%%%%%%%%%%%%%%%%%%%%
%%%%%%%%%%%%%%%%%%%%%%%%%%%%%%%%%%

%%%%%%%%%%%%%%55

\end{document}